\documentclass{amsart}
\usepackage[body={6in, 8in}, a4paper]{geometry}
\usepackage[all]{xy}
\usepackage{amssymb, amsmath, color, hyperref}
\hypersetup{colorlinks=true}

\makeatletter
\def\url@leostyle{%
  \@ifundefined{selectfont}{\def\UrlFont{\sf}}{\def\UrlFont{\small\ttfamily}}}
\makeatother
\definecolor{dblue}{rgb}{0,0,0.7}
\newtheoremstyle{mythm}{11pt}{11pt}{\it\color{dblue}}{}{\bf\color{dblue}}{.}{ }{}
\theoremstyle{mythm}
\newtheorem{theorem}{Theorem}[section]
\newtheorem{lemma}[theorem]{Lemma}
\newtheorem{corollary}[theorem]{Corollary}
\newtheorem{proposition}[theorem]{Proposition}

\theoremstyle{definition}

\newtheoremstyle{myremark}{3pt}{3pt}{}{}{\bfseries}{.}{ }{}
\theoremstyle{myremark}
\newtheorem{remark}[theorem]{Remark}
\newtheorem{example}[theorem]{Example}

\font\russ=wncyr10  1
\def\sha{\hbox{\russ\char88}}

\newcommand{\eins}{\boldsymbol{1}}

\DeclareMathOperator{\Ext}{Ext}
\DeclareMathOperator{\Ind}{Ind}

\DeclareMathOperator{\Inf}{Inf}
\DeclareMathOperator{\Gal}{Gal}
\DeclareMathOperator{\Hom}{Hom}

\DeclareMathOperator{\res}{res}

\DeclareMathOperator{\Sel}{Sel}

\DeclareMathOperator{\rk}{rk}
\DeclareMathOperator{\srk}{srk}

\DeclareMathOperator{\cok}{cok}

\newcommand{\QQ}{\mathbb{Q}}

\newcommand{\ZZ}{\mathbb{Z}}

\newcommand{\Ir}{\mathrm{Ir}}

\newcounter{condone}

\newenvironment{conditions}{\begin{list}{(\alph{condone})}{\usecounter{condone}}}{\end{list}}

\begin{document}

\title[Structure of Selmer groups]{On the Galois structure of Selmer groups}

\author{David Burns, Daniel Macias Castillo and Christian Wuthrich}

\address{King's College London, Department of Mathematics, London WC2R 2LS,
U.K.}
\email{david.burns@kcl.ac.uk}

\address{Instituto de Ciencias Matem\'aticas (ICMAT), 28049 Madrid, Spain.}
\email{daniel.macias@icmat.es}

\address{School of Math. Sciences,
University of Nottingham,
Nottingham NG7 2RD,
U.K.}
\email{christian.wuthrich@nottingham.ac.uk}

\begin{abstract} Let $A$ be an abelian variety defined over a number field $k$ and $F$ a finite Galois extension of $k$. Let $p$ be a prime number. Then under certain not-too-stringent
conditions on $A$ and $F$ we investigate the explicit Galois structure of the $p$-primary Selmer group of $A$ over $F$. We also use the results so obtained to derive new
bounds on the growth of the Selmer rank of $A$ over extensions of $k$.
\end{abstract}

\maketitle

\section{Introduction}\label{Intro}

Let $A$ be an abelian variety defined over a number field $k$ and write $A^t$ for the corresponding dual abelian variety.

Fix a prime number $p$. For each finite extension field $F$ of $k$ we write $\Sel_p(A/F)$ for the $p$-primary Selmer group  of $A$ over $F$.
At the outset we recall that this group can be defined as the direct limit, over non-negative integers $n$, of the Selmer groups associated to the isogenies $[p^n]$ of $A$ over $F$. It is then equal to the subgroup of the Galois cohomology group $H^1\bigl(F,A[p^\infty]\bigr)$ given by the usual local conditions, where $A[p^\infty]$ denotes
the Galois module of $p$-power torsion points on $A$.

Write $X(A/F)$ for the Pontryagin dual of  $\Sel_p(A/F)$. Denote by $T(A/F)$ the torsion subgroup of $X(A/F)$ and by $\overline{X}(A/F)$ the quotient of $X(A/F)$ by $T(A/F)$. We recall that  $X(A/F)$ contains a subgroup canonically isomorphic to the $p$-primary Tate-Shafarevich group $\sha(A^t/F)_p$ of $A^t$ over $F$, with the associated quotient group canonically isomorphic to $\Hom_\ZZ(A(F),\ZZ_p)$.
In particular, if $\sha(A/F)_p$ is finite, then $T(A/F)$ and
$\overline{X}(A/F)$ simply identify with $\sha(A^t/F)_p$ and $\Hom_\ZZ\bigl(A(F),\ZZ_p\bigr)$ respectively.

Let now $F/k$ be a Galois extension of group $G$. In this case we wish to study the structure of $X(A/F)$ as a $G$-module. We recall that describing the explicit Krull-Schmidt decomposition of $\ZZ_p[G]$-lattices that occur naturally in arithmetic is known to be
a very difficult problem (see, for example, the considerable difficulties already encountered by Rzedowski-Calder\'on et al in~\cite{RVM} when considering the pro-$p$ completion of the ring of algebraic
integers of $F$).

Notwithstanding this fact, in this article we will show that under certain not-too-stringent conditions on $A$ and $F$, there is a very strong interplay
between the structures of the $G$-modules $T(A/F)$ and $\overline{X}(A/{F})$ and that this interplay can in turn lead to concrete results about the explicit
structure of $\overline{X}(A/F)$ as a $\ZZ_p[G]$-module. For example, one of our results characterises the projectivity of $\overline{X}(A/{F})$ as a
$\ZZ_p[G]$-module (see Theorem \ref{noether} and Remark~\ref{noether rem}), whilst another characterises, for certain $G$, the conditions under which
$\overline{X}(A/{F})$ is a
 trivial source module over $\ZZ_p[G]$ (see Theorem~\ref{yakovlev}). More generally, our methods can be used in some cases to give an explicit decomposition of
 $\overline{X}(A/F)$ as a direct sum of indecomposable $\ZZ_p[G]$-modules (see the examples in \S\ref{examples}).

Whilst such explicit structure results might perhaps be of some intrinsic interest there are also two ways in which they can have important consequences.

Firstly, they give additional information about the change of rank of $X(A/L)$ as $L$ varies over the intermediate fields of $F/k$. As a particular example of this, we show here that this information can in some cases be used to tighten the bounds on the growth of ranks of Selmer groups in dihedral towers of numbers fields that were proved by Mazur and Rubin in \cite{mr2} (see Corollary \ref{mr cor}). In a similar way we also show that it sheds new light on Selmer ranks in false Tate curve towers of number fields (see Corollary \ref{ftct cor}).

Secondly, such structure results play an essential role in attempts to understand and investigate certain equivariant refinements of the Birch and Swinnerton-Dyer conjecture. In this regard, we note that in \cite{mw} the structure results proved here play a key role in our obtaining the first (either theoretical or numerical) verifications of the $p$-part of the equivariant Tamagawa number conjecture for elliptic curves $A$ in the technically most demanding case in which $A$ has strictly positive rank over $F$ and the Galois group $G$ is both non-abelian and of exponent divisible by (an arbitrarily large power of) $p$.

Finally we note that some of our structure results concerning projectivity are similar in spirit to the results on Selmer groups of elliptic curves over cyclotomic $\ZZ_p$-extensions that are proved by Greenberg in~\cite{RG2} (for more details in this regard see Remark ~\ref{noether rem}).

\section{Statements of the results}\label{mwts}

In this section we state all of the main results that are to be proved in \S\ref{Proof of the main structure results}.

\subsection{Basic notations}
For any finite group $\Gamma$ we write $\Ir(\Gamma)$ for the set of irreducible $\QQ_p^c$-valued characters of $\Gamma$. For each $\psi$ in $\Ir(\Gamma)$ we write
$$
 e_\psi = \frac{\psi(1)}{|\Gamma|}\, \sum_{\gamma \in \Gamma}\psi(\gamma^{-1})\,\gamma
$$
for the primitive idempotent of the centre $\zeta\bigl(\QQ_p^c[\Gamma]\bigr)$ of the group ring $\QQ_p^c[\Gamma]$. We fix a $\QQ_p^c[\Gamma]$-module $V_\psi$ of character $\psi$. We also write $\eins_\Gamma$ for the trivial character of $\Gamma$.

We also note that by a `$\ZZ_p[\Gamma]$-lattice' we shall mean a $\ZZ_p[\Gamma]$-module that is both finitely generated and free over $\ZZ_p$.

For any abelian group $M$ we write $M_{\rm tor}$ for its torsion subgroup. For any prime $p$ we write  $M[p]$ for the subgroup $\{m \in M: pm =0\}$ of $M_{\rm tor}$ and set $M_p := \ZZ_p\otimes_\ZZ M$. For any finitely generated $\ZZ_p$-module $M$ and any field extension $F$ of $\QQ_p$ we set $F\otimes M := F\otimes_{\ZZ_p}M$ and write $\rk_{\ZZ_p}(M)$ for the $\ZZ_p$-rank $\dim_{\QQ_p}(\QQ_p\otimes M)$ of $M$.

For any Galois extension of fields $L/K$ we write $G_{L/K}$ in place of $\Gal(L/K)$. For each non-archimedean place $v$ of a number field we write $\kappa_v$ for its residue field. We also write $\mathbb{F}_p$ for the field with $p$ elements.

\subsection{The hypotheses}\label{Preliminaries} We assume throughout to be given a finite Galois extension of number fields $F/k$ of group $G$ and an abelian variety $A$ defined over $k$. We also fix a prime number $p$ and a $p$-Sylow subgroup $P$ of $G$ and then set $K := F^P$.

In this section we assume in addition that $A$ and $F$ satisfy the following hypotheses.
\begin{conditions}
\item\label{hyp_a}   $A(K)[p] = 0$;
\item\label{hyp_b}  No Tamagawa number of $A_{/K}$ is divisible by $p$;
\item\label{hyp_c} $A_{/K}$ has good reduction at all $p$-adic places;
\item\label{hyp_d}  For all $p$-adic places $v$ that ramify in $F/K$, the reduction is ordinary and $A(\kappa_v)[p]= 0$;
\item\label{hyp_e}   No place of bad reduction for $A_{/k}$ is ramified in $F/k$.
\end{conditions}

\begin{remark}
 These hypotheses are for the most part motivated by the arguments that are used by Greenberg in~\cite{RG} (see, in particular, the proof of Proposition~\ref{sha tate} below). For a fixed abelian variety $A$ over $k$ the first three hypotheses are clearly satisfied by all but finitely many primes $p$. However, the hypothesis~\ref{hyp_d} excludes the case that is called `anomalous' by Mazur in~\cite{m} and, for a given $A$, there may be infinitely many primes $p$ for which there are $p$-adic places $v$ at which $A$ has good ordinary reduction but $A(\kappa_v)[p]$ does not vanish.
Nevertheless, it is straightforward to describe examples of abelian varieties $A$ for which there are only finitely many such anomalous places -- see, for example, the result of Mazur and Rubin in~\cite[Lemma A.5]{mr}.
\end{remark}

\subsection{The Galois structure of \texorpdfstring{$X(A/F)$}{X(A/F)}}\label{Statement of the main structure results}
 We first recall from the introduction that $X(A/F)$ is defined as the Pontryagin dual of $\Sel_p(A/F)$ and that,
 if the group $\sha(A/F)_p$ is finite, then the tautological short exact sequence
 \[0\to T(A/F)\to X(A/F)\to \overline{X}(A/F)\to 0\] simply identifies with the canonical short exact sequence
 \[0\to \sha(A^t/F)_p\to X(A/F)\to \Hom_\ZZ(A(F),\ZZ_p)\to 0.\]
 
 For each $\rho$ in $\Ir(P)$ we write $m_\rho$, or $m_{F,\rho}$ when we wish to be more precise,
 for the multiplicity with which $\rho$ occurs in the representation $\QQ_p^c\otimes X(A/F)$. For each subfield $L$ of $F$ we set
\[ \srk(A/L) := \dim_{\QQ_p}(\QQ_p\otimes X(A/L)).\]
For any fields $L$ and $L'$ with $k \subseteq L \subseteq L' \subseteq F$ we also set
\[
T^{L'}\!\!(A/L) := \cok\Bigl(T(A/L') \xrightarrow{\pi^{L'}_L} T(A/L)\Bigr),
\]
where $\pi^{L'}_L$ denotes the natural norm map.

\subsubsection{Projectivity results} Our first results give an explicit characterisation of the conditions under which $\overline{X}(A/F)$ can be a projective
$\ZZ_p[G]$-module.

\begin{theorem}\label{noether}
  If $A$ and $F$ satisfy hypotheses~\ref{hyp_a}--\ref{hyp_e}, then the following conditions are equivalent.
  \begin{enumerate}
  \item[(i)]\label{noether_i}   $\overline{X}(A/F)$ is a projective $\ZZ_p[G]$-module.
  \item[(ii)]\label{noether_ii}  $T^F\!(A/F^H) =  0 $ for $H\leq P$ and $X(A/F)$ spans a free $\QQ_p[P]$-module.
  \item[(iii)]\label{noether_iii} $T^F\!(A/F^H) =  0 $ for $H\leq P$ and $m_{\rho} \ge \rho(1)\cdot\srk(A/K)$ for $\rho\in \Ir(P)$.
  \item[(iv)]\label{noether_iv}  For each $H\leq P$ both $T^F\!(A/F^H) =  0 $ and $\srk(A/F) = |H|\cdot \srk(A/F^H)$.
  \item[(v)]\label{noether_v}   For each $H\leq P$ with $|H|\! =\! p$ both $T^F(A/F^H)\! =\!0 $ and $\srk(A/F) = p\cdot \srk(A/F^H)$.
  \end{enumerate}
\end{theorem}

\begin{remark}
 The equivalence of the conditions (i) and (v) in Theorem \ref{noether} is somewhat analogous to a classical theorem of Noether. To see this write $\mathcal{O}_{L}$ for the ring of integers of a number field $L$ and $\kappa_{L,p}$ for the direct sum of its residue fields at all $p$-adic places. Then for $H\leq P$ with $|H| = p$ one has
$\rk(\mathcal{O}_F) = p\cdot \rk(\mathcal{O}_{F^H})$ (by Galois theory) and Noether's Theorem \cite{no} asserts that $\mathcal{O}_{F,p}$ is a projective $\ZZ_p[G]$-module if and only if for each such $H$ the cokernel of the trace map $\kappa_{F,p} \to \kappa_{F^H,p}$ vanishes.
\end{remark}

\begin{remark}\label{noether rem}
 Some aspects of Theorem~\ref{noether} are also reminiscent of, but differ in some important respects from, results that are proved by Greenberg in~\cite{RG2}. To be more precise, let $F_\infty$ and $k_\infty$ be the cyclotomic $\ZZ_p$-extensions of $F$ and $k$, set $\Delta := G_{F_\infty/k_\infty}$,
 let $A$ be an elliptic curve over $k$ that has good ordinary reduction at all $p$-adic places and for any extension $L$ of $k$ in $F_\infty$ set
 $X(A/L) := \Hom_{\ZZ_p}\bigl(\Sel_p(A/L),\QQ_p/\ZZ_p\bigr)$. Then, assuming crucially that $X(A/F_\infty)$ is both torsion over the relevant Iwasawa algebra and has
 zero $\mu$-invariant, Greenberg investigates the explicit structure of $X(A/F_\infty)$ as a $\ZZ_p[\Delta]$-module.
In particular, by combining the same characterisation of projectivity in terms of cohomological triviality (that we use in~\S\ref{proof of noether} to prove
Theorem~\ref{noether}) together with certain intricate computations in Galois cohomology, he proves under certain mild additional conditions on $A$ and $F$ that
an `imprimitive' form of $X(A/F_\infty)$ is projective over $\ZZ_p[\Delta]$ and then deduces families of explicit relations between the multiplicities with
which each $\psi \in \Ir(\Delta)$ occurs in $\QQ_p^c\otimes X(A/F_\infty)$.
However, even if the natural descent homomorphism $\QQ_p\otimes X(A/F_\infty)_{G_{F_\infty/F}} \to \QQ_p\otimes X(A/F)$ is bijective, in order to deduce results about
the structure of the $\QQ_p[G]$-module $\QQ_p\otimes X(A/F)$ one would need to describe $\QQ_p\otimes X(A/F_\infty)_{G_{F_\infty/F}}$ as an explicit quotient
of the $\QQ_p[\Delta]$-module $\QQ_p\otimes X(A/F_\infty)$ and, in general, this appears to be difficult.
\end{remark}

Regarding Theorem~\ref{noether}(v) we note that if $P$ is either cyclic or generalised quaternion (so $p=2$), then it has a unique subgroup of order $p$.

In another direction, if we assume the vanishing of $\sha(A/K)_p$, then we are led to the following simplification of Theorem~\ref{noether}. In the subsequent article \cite{mw} we use this particular characterisation in order to obtain important consequences in the context of equivariant refinements of the Birch and Swinnerton-Dyer conjecture.

\begin{corollary}\label{noether cor}
If $A$ and $F$ satisfy the hypotheses~\ref{hyp_a}--\ref{hyp_e} and $\sha(A/K)_p$ vanishes, then the following conditions are equivalent.
  \begin{enumerate}
  \item[(i)]\label{noether cor_i}   $X(A/F)$ is a projective $\ZZ_p[G]$-module.
  \item[(ii)]\label{noether cor_ii}  $\QQ_p\otimes X(A/F)$ is a free $\QQ_p[P]$-module.
  \item[(iii)]\label{noether cor_iii} $m_{\rho} \ge \rho(1)\cdot\srk(A/K)$ for $\rho\in \Ir(P)$.
  \item[(iv)]\label{noether cor_iv}  For each $H\le P$ both $T(A/F^H) =  0 $ and $\srk(A/F) = |H|\cdot \srk(A/F^H)$.
  \item[(v)]\label{noether cor_v}   For each $H\leq P$ with $|H|\! =\! p$ both $T(A/F^H)\! =\! 0 $ and $\srk(A/F) = p\cdot \srk(A/F^H)$.
  \end{enumerate}
\end{corollary}

Since obtaining an explicit description (in terms of generators and relations) of any given $\ZZ_p[G]$-module is in general very difficult to achieve, the simplicity of the equivalence of the conditions (i) and (ii) in Corollary~\ref{noether cor} seems  striking. In addition, the result of Theorem~\ref{upper bound} below will show that the value of $\srk(A/F)$ that occurs in Theorem~\ref{noether}(iv) and Corollary~\ref{noether cor}(iv) is actually the maximal possible in this context.

\subsubsection{General results} In the next result we provide evidence that there is a strong link between the Galois structures of Mordell-Weil and Tate-Shafarevich
groups under much more general hypotheses than occur in Theorem~\ref{noether} and Corollary~\ref{noether cor}.

For any subgroup $H$ of $G$ we write $\ZZ_p[G/H]$ for $\ZZ_p[G]\otimes_{\ZZ_p[H]}\ZZ_p$, regarded as a left $\ZZ_p[G]$-module in the obvious way.
Given a subgroup $J$ of $G$, by a `trivial source
$\ZZ_p[G]$-module with vertices contained in $J$' we shall then mean a direct sum of finitely many (left) $\ZZ_p[G]$-modules, each of which is isomorphic to a direct
summand of $\ZZ_p[G/H]$ for some
subgroup $H$ of $J$. In particular, by a `trivial source $\ZZ_p[G]$-module' we shall mean a trivial source $\ZZ_p[G]$-module with vertices contained in $G$.
We also note that if $G$ contains a normal subgroup $I$ of $p$-power index, then by a result of Berman and Dress (cf.~\cite[Th. 32.14]{curtisr}) one knows that
$\ZZ_p[G/H]$ is an indecomposable $\ZZ_p[G]$-module for any subgroup $H$ of $G$ that contains $I$. This fact will be useful in the sequel.

\begin{remark}\label{inconsistency} We wish to point out that the terminology introduced above differs from the one used in \cite{omac}, where trivial
source modules are instead called `permutation modules' (see \S 2.2 in loc. cit.). In particular, under our current terminology and in the context of
dihedral extensions of $\QQ$, the relevant assertion of \cite[Corollary 5.3 (ii)]{omac} would state that the modules $A^0(T_{F})$ and ${\rm Sel}^0(T_{F})_{\rm tf}$
defined in loc. cit. are trivial source $\ZZ_p[G]$-modules. 
\end{remark}

\begin{theorem}\label{yakovlev}
 Assume that $A$ and $F$ satisfy~\ref{hyp_a}--\ref{hyp_e} and that $P$ is cyclic of order $p^n$. For each integer $i$ with $0 \le i\le n$, let $F_i$ denote the unique
 field with $K \subseteq F_i \subseteq F$ and $[F:F_i]=p^{n-i}$. Then the following claims are valid.

\begin{enumerate}
\item[(i)]\label{yakovlev_i} The $\ZZ_p[P]$-module $\overline{X}(A/F)$ is determined up to isomorphism by the integers $\srk(A/F_i)$ for each $i$ with $0\le i \le n$
together with the diagram of torsion groups
\begin{equation}\label{yak diag}
T^F\!(A/F_{0}) \ \leftrightarrows\ T^F\!(A/F_{1})\ \leftrightarrows \ \cdots \ \leftrightarrows\  T^F\!(A/F_{n-1})
\end{equation}
where the upper and lower arrows are induced by the natural corestriction and restriction maps (for more details see~\S\ref{proof yakovlev} below).

\item[(ii)]\label{yakovlev_ii} Moreover, if $T^F\!(A/{F_{i}})$ vanishes for each $i$ with $0 \le i < n$, then $\overline{X}(A/F)$ is a
trivial source $\ZZ_p[G]$-module with vertices contained in $P$. In particular, in this case $\overline{X}(A/F)$ is a
trivial source $\ZZ_p[G]$-module, and hence also a trivial source $\ZZ_p[P]$-module.
\end{enumerate}
\end{theorem}

In the context of Theorem \ref{yakovlev} it is worth bearing in mind that even if $G$ is cyclic of $p$-power order (and so $G =P$) the category of $\ZZ_p[G]$-lattices is in general very complicated. For example,  Heller and Reiner \cite{hr} have shown that in this case the number of isomorphism classes of indecomposable $\ZZ_p[G]$-lattices is infinite unless $|G|$ divides $p^2$ and have only been able to give an explicit classification of all such modules in the special cases $|G|=p$ and $|G|=p^2$.

Despite these difficulties, in \S\ref{examples} we will show that in certain cases the classification results of Heller and Reiner can be combined with our approach to make much more explicit the interplay between $\overline{X}(A/F)$ and the groups $T^F\!(A/{F_i})$ that is described in Theorem \ref{yakovlev}.

In a more general direction, we note that if $A$ and $F$ satisfy the hypotheses~\ref{hyp_a}--\ref{hyp_e} then, irrespective of the abstract structure of $P$, the result of Lemma~\ref{sha tate2} below shows that Theorem~\ref{yakovlev} can be applied to Galois extensions $F'/k'$ with $k \subseteq k'\subseteq F'\subseteq F$ and for which $G_{F'/k'}$ has cyclic Sylow $p$-subgroups.
 This shows that in all cases the behaviour of the natural restriction and corestriction maps on $T(A/L)$ as $L$ varies over all intermediate fields of $F/k$ imposes a strong restriction on the explicit structure of the $\ZZ_p[G]$-module $X(A/F)$.

 The following result gives an explicit example of this phenomenon (and for more details in this regard see Remark~\ref{elder} below). We recall that
 a finite group $\Gamma$ is defined to be supersolvable if it has a normal series in which all subgroups are normal in $\Gamma$ and all factors are cyclic groups.

\begin{corollary}\label{general yak}  Assume that $A_{/k}$ has good reduction at all $p$-adic places and ordinary reduction at all $p$-adic places that ramify in $F/k$. Assume also that $A$ and $F$ satisfy the hypotheses~\ref{hyp_a}--\ref{hyp_e}, that $P$ is both abelian and normal and that the quotient $G/P$ is supersolvable.

Then for each irreducible $\QQ_p^c$-valued character $\rho$ of $G$ there is a diagram of torsion groups of the form~\eqref{yak diag} which determines the multiplicity with which $\rho$ occurs in the $\QQ_p^c[G]$-module $\QQ_p^c\otimes X(A/F)$ up to an error of at most $\srk(A/{K})$.

In particular, if the group $\Sel_p(A/K)$ is finite, then the structure of the $\QQ_p[G]$-module $\QQ_p\otimes X(A/F)$ is uniquely determined by the torsion
groups $T(A/L)$ as $L$ ranges over intermediate fields of $F/k$ together with the natural restriction and corestriction maps between them.
\end{corollary}

\subsection{Some consequences for ranks}\label{rank consequences}

In this section we describe several explicit consequences concerning the ranks of Selmer modules which follow from the structure results stated above.

In the first such result we do not restrict the abstract structure of $G$.

\begin{theorem}\label{upper bound}
Assume that $A$ and $F$ satisfy the hypotheses~\ref{hyp_a}--\ref{hyp_e}.
\begin{enumerate}
\item[(i)]\label{upper bound_i}
  Fix a chain of fields $K = F_0\subset F_1 \subset \dots \subset F_n = F$ with each extension $F_i/F_{i-1}$ Galois of degree $p$ and for each index $i$ define an integer $d_i$ by the equality $p^{d_i} = |T^{F_i}(A/{F_{i-1}})|$. Then one has
  \begin{equation*}
    \srk(A/K) + (p-1)\sum_{i=1}^{i=n} d_i \le \srk(A/F) \le [F:K] \cdot \srk(A/K) + (p-1)\sum_{i=1}^{i=n} p^{n-i} d_i.
  \end{equation*}
  \item[(ii)]\label{upper bound_ii}
    Fix $\rho\in \Ir(P)$. If $T^{F^{\ker(\rho)}}\!(A/{F^H})$ vanishes for all $H$ with $\ker(\rho) < H \le P$, then $m_{\rho} \le \rho(1)\cdot\srk(A/K)$.
  \item[(iii)]\label{upper bound_iii}
    If $T^F(A/{F^H})$ vanishes for all $H\leq P$, then $\srk(A/F) \le [F:K]\cdot\srk(A/K)$.
  \end{enumerate}
\end{theorem}

In certain cases the additional hypotheses in Theorem~\ref{upper bound} can be simplified. For example, if $P$ is abelian of exponent $p$, then for any non-trivial
$\rho$ in $\Ir(P)$ the inclusion $K \subseteq L\subsetneq F^{\ker(\rho)}$ implies $L = K$ and so the hypothesis in claim (ii) is satisfied if
$T^{F^{\ker(\rho)}}\!(A/K)$ vanishes and hence, a fortiori, if $\sha(A/K)_p$ vanishes. Similarly, if $|P|=p$ then the hypothesis in claim (iii) is satisfied if $\sha(A/K)_p$ vanishes.

In the remainder of this section we specialise the extension $F/k$ and variety $A$ in order to obtain finer bounds on Selmer rank than are given in Theorem~\ref{upper bound}.

\subsubsection{Dihedral extensions} In the next result we focus on the case that the Sylow $p$-subgroup $P$ is an abelian (normal) subgroup of $G$ of index two and that the conjugation action of any lift to $G$ of the generator of $G/P$ inverts elements of $P$. Using the terminology of Mazur and Rubin~\cite{mr2}, we therefore say that the pair $(G,P)$ is `dihedral'.

This result tightens the bounds on ranks of Selmer modules in dihedral towers of numbers fields that are proved by Mazur and Rubin in \cite{mr2}.

\begin{corollary}\label{mr cor}
 Assume that $p$ is odd, that $A$ is an elliptic curve for which $A$ and $F$ satisfy the hypotheses~\ref{hyp_a}--\ref{hyp_e} and that the pair $(G,P)$ is dihedral (in the above sense). Assume also that all $p$-adic places of $k$ split in the quadratic extension $K/k$. Then the following claims are valid.

\begin{enumerate}
  \item[(i)]\label{mr cor_ii}  If $\srk(A/K)$ is odd, $P$ is cyclic and $T^F(A/{F^H})$ vanishes for all subgroups $H$ of $P$, then $\overline{X}(A/F)$ is a
trivial source $\ZZ_p[G]$-module with vertices contained in $P$ that has a non-zero projective direct summand. In particular, in this case for all characters $\rho$ in $\Ir(P)$ one has $1 \le m_\rho \le \srk(A/K)$.

  \item[(ii)]\label{mr cor_i} If $\srk(A/K) = 1$ and $\sha(A/K)_p$ vanishes, then $\sha(A/{F^H})_p$ vanishes for all subgroups $H$ of $P$, $X(A/F)$ is a projective $\ZZ_p[G]$-module and $m_\rho = 1$ for all characters $\rho$ in $\Ir(P)$.
\end{enumerate}
\end{corollary}

\subsubsection{False Tate curve towers} A `false Tate curve tower' is obtained as the union over natural numbers $m$ and $n$ of fields $K_d^{m,n}:= \QQ\bigl(\sqrt[p^m]{d},\zeta_{p^{n}}\bigr)$ where $d$ is a fixed integer, $m\le n$ and $\zeta_{p^{n}}$ is a choice of primitive $p^{n}$-th roots of unity in $\QQ^c$. Such extensions have recently been investigated in the context of non-commutative Iwasawa theory.

Since $K_d^{m,n}$ is a bicyclic $p$-extension of $\QQ(\zeta_p)$ there are several ways in which the results stated above can be used to restrict the structure of $X(A/{K_d^{m,n}})$.

In the following result we give an example of this phenomenon.

\begin{corollary}\label{ftct cor} Fix a natural number $d$. Let $A$ be an abelian variety over $\QQ$ that has good reduction at all prime divisors of $d$ and good ordinary reduction at $p$. Assume that $A$ has no point of order $p$ over $\QQ(\zeta_p)$, that the reduction of $A$ at $p$ has no point of order $p$ over $\mathbb{F}_p$ and that no Tamagawa number of $A$ over $\QQ(\zeta_p)$ is divisible by $p$.

Let $n$ be any natural number for which $\sha\bigl(A/{\QQ(\zeta_{p^{n}})}\bigr)_p$ vanishes and set
\[ n_d := \begin{cases} &\text{least integer $m$ with both $0\le m \le n$ and $\sha(A/{K_d^{m,n}})_p\not= 0$},\\
                        &n \,\,\,\,\text{if no such integer $m$ exists.}\end{cases}\]
Write $P_{n}$ and $G_{n}$ for the Galois groups of $K_d^{n_d,n}$ over $\QQ(\zeta_{p^n})$ and over the subextension of $\QQ(\zeta_{p^n})$ of degree $p^{n-1}$ over $\QQ$ respectively. Then the following claims are valid.

\begin{itemize}
\item[(i)] $\overline{X}(A/{K_d^{n_d,n}})$ is a
trivial source $\ZZ_p[G_n]$-module with vertices contained in $P_n$.
\item[(ii)] The multiplicity $m_\rho$ with which each $\rho$ in $\Ir(G_{n})$ occurs in the module $\QQ^c_p\otimes X(A/{K_d^{n_d,n}})$ is at most $\srk\bigl(A/{\QQ(\zeta_{p^n})}\bigr)$. Further if $\rho$ and $\rho'$ are any non-linear characters in $\Ir(G_{n})$, then $m_\rho \le m_{\rho'}$ if and only if $\ker(\rho) \subseteq \ker(\rho')$.
\item[(iii)] For any complementary subgroup $H$ to $P_{n}$ in $G_{n}$ and any $\phi$ in $\Ir(H)$ one has $\rk_{\ZZ_p}(e_\phi \cdot X(A/{K_d^{n_d,n}})) \le p^{n_d}\cdot \rk_{\ZZ_p}\bigr(e_\phi \cdot X(A/{\QQ(\zeta_{p^n})})\bigr)$.
    \end{itemize}\end{corollary}

\section{Proofs}\label{Proof of the main structure results}

In this section we prove all of the results that are stated above. More precisely, after making some important general observations in \S\ref{ago} we prove Theorem \ref{yakovlev} and Corollary \ref{general yak} in \S\ref{proof yakovlev}, Theorem \ref{upper bound} in \S\ref{proof of upper bound}, Theorem \ref{noether} in \S\ref{proof of noether}, Corollary \ref{noether cor} in \S\ref{proof of noether cor} and then finally Corollaries \ref{mr cor} and \ref{ftct cor} in \S\ref{last}.

\subsection{General observations}\label{ago}

In the following result we write $N_G(J)$ for the normaliser of a subgroup $J$ of $G$.

\begin{proposition}\label{sha tate}
 Assume that $A$ and $F$ satisfy the hypotheses ~\ref{hyp_a}--\ref{hyp_e}. Then for each subgroup $J$ of $P$ the natural norm map $X(A/{F})_J \to X(A/{F^J})$ is bijective and the group $T^{F}(A/{F^J})$ is isomorphic as a $\ZZ_p\bigl[N_G(J)\bigr]$-module to the Tate cohomology group $\hat H^{-1}\bigl(J,\overline{X}(A/F)\bigr)$.
\end{proposition}
\begin{proof}
We first note that, if we strengthen condition~\ref{hyp_d} to impose that the reduction is good ordinary and non-anomalous at all places above $p$, then our conditions
coincide precisely with the hypotheses of Proposition~5.6 used by Greenberg in~\cite{RG}. In this case the restriction homomorphism
$\res^{F^J}_F: \Sel_p(A/{F^J}) \to \Sel_p(A/{F})^J$ is therefore bijective.
However, the proof of Proposition~5.6 and the result in Proposition~4.3 in~\cite{RG} show that $\res^{F^J}_F$ is also still bijective if there are places above $p$
which do not ramify in $F/K$ and for which the reduction is good but not necessarily ordinary or non-anomalous and so the restriction map
$\Sel_p(A/{F^J}) \to \Sel_p(A/{F})^J$ is bijective under the given hypotheses. By taking Pontryagin duals this implies that the norm map
$X(A/{F})_J \to X(A/{F^J})$ is bijective and hence restricts on torsion subgroups to give an isomorphism
$X(A/{F})_{J,{\rm tor}} \cong X(A/{F^J})_{\rm tor} = T(A/{F^J})$.

Upon taking $J$-coinvariants of the tautological exact sequence
\begin{equation*}\label{taut} 0 \to T(A/{F}) \to X(A/F) \to \overline{X}(A/F)\to 0\end{equation*}
and then passing to torsion subgroups in the resulting exact sequence we therefore obtain an exact sequence of $\ZZ_p[N_G(J)]$-modules
\begin{equation*}\label{coinvariance}  T(A/F) \xrightarrow{\pi^F_{F^J}} T(A/{F^J}) \to \overline{X}(A/F)_{J,{\rm tor}} \to 0.\end{equation*}

Given this exact sequence and the definition of $T^F(A/{F^J})$ as the cokernel of $\pi^F_{F^J}$, the claimed isomorphism is a direct consequence of the fact that
$\overline{X}(A/F)_{J,{\rm tor}}$ is equal to $\hat H^{-1}\bigl(J,\overline{X}(A/F)\bigr)$. Indeed, since  $\hat H^{-1}\bigl(J,\overline{X}(A/F)\bigr)$ is finite and
$\overline{X}(A/F)$ is $\ZZ_p$-free, the latter equality follows immediately from the tautological exact sequence
\begin{equation}\label{taut seq} 0 \to \hat H^{-1}\bigl(J,\overline{X}(A/F)\bigr) \xrightarrow{\subseteq} \overline{X}(A/F)_{J} \to \overline{X}(A/F)^J \to \hat H^0\bigl(J,\overline{X}(A/F)\bigr)\to 0\end{equation}
where the third arrow is induced by the action of the element $\sum_{g \in J}g$ of $\ZZ_p[G]$.

\end{proof}

\begin{remark} If in the setting of Proposition \ref{sha tate} one has $\srk(A/F) = \srk(A/F^J)$, then $\overline{X}(A/F)$ is isomorphic as a $\ZZ_p[J]$-module to $\ZZ_p^{\srk(A/F)}$ and so the group $\hat H^{-1}\bigl(J,\overline{X}(A/F)\bigr)$ vanishes. Proposition \ref{sha tate} therefore implies that in this case the group $T^{F}(A/F^J)$ also vanishes and hence, a fortiori, that the order of $T(A/F)$ is divisible by the order of $T(A/{F^J})$. This shows that, under the stated hypotheses, Proposition \ref{sha tate} provides a useful quantitative measure of the general principle that any decrease in the order of Tate-Shafarevich groups must be accompanied by a corresponding increase in the rank of Mordell-Weil groups. For details of some explicit applications of this result see \S\ref{examples}.\end{remark}

\begin{remark} Since the Tate cohomology groups of $\overline{X}(A/F)$ are finite, for any subgroup $J$ of $P$ the isomorphism
$X(A/F)_J \cong X(A/{F^J})$ in Proposition \ref{sha tate} combines with the exact sequence~\eqref{taut seq} to imply that
\begin{equation}\label{rationalisom}\QQ_p\otimes X(A/{F^J})\cong\QQ_p\otimes X(A/F)^J\end{equation} and in particular that
\[\srk(A/{F^J})=\rk_{\ZZ_p}\bigl(\overline{X}(A/F)^J\bigr).\] In fact, these statements at the rational level hold true in full generality and do not depend on the arguments
or hypotheses required to prove the integral assertions of Proposition \ref{sha tate}. For convenience and brevity,
we shall often in the sequel use these facts without further explicit comment. \end{remark}

To end this section we prove a useful technical result concerning the hypotheses~\ref{hyp_a}-\ref{hyp_e}.

\begin{lemma}\label{sha tate2} Let $k', K'$ and $F'$ be fields with $k \subseteq k'\subseteq K'\subseteq F' \subseteq F$ and such that $F'/k'$ is Galois and $G_{F'/K'}$
is a Sylow $p$-subgroup of $G_{F'/k'}$.

\begin{enumerate}
\item[(i)] Assume that $A_{/k}$ and $F$ satisfy the hypotheses~\ref{hyp_a}, \ref{hyp_b}, \ref{hyp_e} and the second part of \ref{hyp_d} (with respect to $K$)
and that $A_{/k}$ has good reduction at all $p$-adic places and
ordinary reduction at all $p$-adic places that ramify in $F/k$. Then $A_{/k'}$ and $F'$ satisfy  the hypotheses \ref{hyp_a}-\ref{hyp_e} with respect to $K'$.

\item[(ii)] If $A_{/k}$ and $F$ satisfy the hypotheses~\ref{hyp_a}--\ref{hyp_e} (with respect to $K$) and $K \subseteq K'$, then $A_{/k'}$ and $F'$ satisfy the
hypotheses~\ref{hyp_a}--\ref{hyp_e} with respect to $K'$.
\end{enumerate}
\end{lemma}
\begin{proof} We note first that the validity of hypotheses \ref{hyp_a} and \ref{hyp_e} for $A_{/k}$ and $F$ implies their validity for $A_{/k'}$ and $F'$. This is clear
for \ref{hyp_e} and for \ref{hyp_a} follows from the fact that $F/K$ is a Galois extension of $p$-power degree and so the vanishing of $A(K)[p]$ implies the vanishing of
$A(F)[p]$ and hence also that of $A(K')[p]$.

To consider hypothesis \ref{hyp_b} we fix a place $v'$ of bad reduction for $A_{/K'}$ and a place $w$ of $F$ above $v'$ and write $v$ for the place of $K$ which lies
beneath $w$. Then hypothesis \ref{hyp_e} implies that $v$ and $v'$ are unramified in $F/K$ and $F/K'$ respectively. Thus, if the Tamagawa number of $A_{/K}$ at $v$ is
coprime to $p$, then so is that of $A_{/F}$ at $w$ since $F/K$ is a Galois extension of $p$-power degree, and hence also that of $A_{/K'}$ at $v'$.

In a similar way one can show that the second part of hypothesis \ref{hyp_d} for $A_{/K}$ and $F$ implies the analogous assertion for $A_{/K'}$ and $F'$. Indeed, if $v'$
is any $p$-adic place of $K'$ which ramifies in $F'$ and $w$ any place of $F$ above $v'$, then the restriction $v$ of $w$ to $K$ is ramified in $F/K$ and hence (by
assumption) $A(\kappa_v)[p]$ vanishes. This implies that $A(\kappa_w)[p]$ vanishes (since $F/K$ is a Galois extension of $p$-power degree) and hence also that
$A(\kappa_{v'})[p]$ vanishes, as required.

In particular, since the validity of hypothesis \ref{hyp_c} and the first part of hypothesis \ref{hyp_d} for $A_{/K'}$ and $F'$ follow directly from the reduction
assumptions that are made in claim (i), that result is now clear.

To discuss claim (ii) we assume that $A_{/k}$ and $F$ satisfy the hypotheses~\ref{hyp_a}--\ref{hyp_e}. We also replace $F'$ by its Galois closure over $K$ (if necessary)
in order to assume, without loss of generality, that $F'/K$ is a Galois extension. Then it is easy to see that the hypotheses~\ref{hyp_c}, \ref{hyp_e} and the first part
of~\ref{hyp_d} will hold for $A_{/k'}$ and $F'$. In addition, since $F'/K$ is a Galois $p$-extension, the same arguments as used to prove claim (i) show that the
hypotheses ~\ref{hyp_a}, \ref{hyp_b} and the second part of~\ref{hyp_d} will also hold for $A_{/K'}$ and $F'$. This proves claim~(ii).\end{proof}

\subsection{The proofs of Theorem \ref{yakovlev} and Corollary \ref{general yak}}\label{proof yakovlev} In this section we make precise and prove all of the claims in Theorem \ref{yakovlev} and Corollary \ref{general yak}.

\subsubsection{The proof of Theorem \ref{yakovlev}} We use the notation and hypotheses of Theorem~\ref{yakovlev}. We set $\overline{X} := \overline{X}(A/F)$ and
for each integer $i$ with $0 \le i \le n$ also $P_i := G_{F/F_i}$, $N_i := N_G(P_i)$ and $\sha_i := T^F(A/{F_{i}})$.

For each such integer $i$, one can check that the natural restriction $\hat H^{-1}\bigl(P_i,\overline{X}\bigr) \to \hat H^{-1}\bigl(P_{i+1},\overline{X}\bigr)$ and
corestriction maps  $\hat H^{-1}\bigl(P_{i+1},\overline{X}\bigr) \to \hat H^{-1}\bigl(P_{i},\overline{X}\bigr)$ on Tate cohomology correspond under
Proposition~\ref{sha tate} to the homomorphisms $\sha_i \to \sha_{i+1}$ and $\sha_{i+1} \to \sha_i$ that are induced by the natural restriction and corestriction
homomorphisms on Galois cohomology.

Given this, the main result of Yakovlev in~\cite{yakovlev} implies that the diagram~\eqref{yak diag} determines the isomorphism class of the $\ZZ_p[G]$-module
$\overline{X}$ up to addition of trivial source $\ZZ_p[G]$-modules with vertices contained in $P$. To be more precise, we recall that~\cite[Theorem 2.4 and Lemma 5.2]{yakovlev}
combine to give the following result:

\begin{lemma}\label{yako}
Assume that $M$ and $M'$ are any finitely generated $\ZZ_p[G]$-modules for which, for each $i$ with $0 \le i < n$, there are isomorphisms of
$\ZZ_p[N_i]$-modules $\kappa_i:\hat H^{-1}(P_{i},M) \to \hat H^{-1}(P_{i},M')$ that lie in commutative diagrams

\begin{equation*}\label{diagrams}\xymatrix{
  \hat H^{-1}(P_{i},M) \ar[r] \ar[d]_{\kappa_i} &  \hat H^{-1}(P_{i+1},M)\ar[d]^{\kappa_{i+1}} \\
 \hat H^{-1}(P_{i},M') \ar[r]&\hat H^{-1}(P_{i+1},M') }\xymatrix{\hat H^{-1}(P_{i},M)\ar[d]_{\kappa_i} & \ar[l] \hat H^{-1}(P_{i+1},M) \ar[d]^{\kappa_{i+1}}\\\hat H^{-1}(P_{i},M') & \ar[l] \hat H^{-1}(P_{i+1},M')}\end{equation*}
where the horizontal arrows are the natural restriction and corestriction homomorphisms. Then there are direct sum decompositions of $\ZZ_p[G]$-modules
\begin{equation}\label{yakovlev decomp}
 M = M_1 \oplus M_2, \qquad\text{ and }\qquad M' = M'_1\oplus M'_2
 \end{equation}
where $M_1$ and $M'_1$ are isomorphic and $M_2$ and $M'_2$ are both trivial source $\ZZ_p[G]$-modules with vertices contained in $P$.
In particular, if $\hat H^{-1}(P_{i},M) =  0 $ for each $i$, then $M$ is itself a
trivial source $\ZZ_p[G]$-module with vertices contained in $P$.\end{lemma}

By simply applying the final assertion of Lemma \ref{yako} to the module $M = \overline{X}$ we complete the proof of Theorem~\ref{yakovlev} (ii).

At this stage, to complete the proof of Theorem~\ref{yakovlev} (i) it suffices to assume that there is a decomposition of $\ZZ_p[G]$-modules
$\overline{X} = M_1\oplus M_2$ where $M_1$ is uniquely determined up to isomorphism and $M_2$ is a trivial source $\ZZ_p[G]$-module, and then to show that
$M_2$ is determined up to isomorphism as a $\ZZ_p[P]$-module by the value of $\srk(A/{F_i})$ for each $i$ with $0 \le i \le n$.
But, it is easy to see that any trivial source $\ZZ_p[P]$-module $W$ is determined uniquely up to isomorphism by the integers $\rk_{\ZZ_p}(W^{P_i})$ for
each $i$ with $0 \le i \le n$, and so the required fact follows directly from the equalities
\[
 \rk_{\ZZ_p}\Bigl(M_2^{P_i}\Bigr) = \rk_{\ZZ_p}\Bigl(\overline{X}^{P_i}\Bigr) - \rk_{\ZZ_p}\Bigl(M_1^{P_i}\Bigr) = \srk(A/{F_i}) - \rk_{\ZZ_p}\Bigl(M_1^{P_i}\Bigr),
\]
where the second equality follows directly from ~\eqref{rationalisom}.

This completes the proof of Theorem~\ref{yakovlev}.

\begin{remark}\label{elder}
If $A$ and $F$ satisfy the hypotheses~\ref{hyp_a}--\ref{hyp_e}, then Lemma~\ref{sha tate2} implies that the result of Theorem~\ref{yakovlev} can be applied to a variety
of Galois extensions $F'/k'$, with $k \subseteq k'\subseteq F'\subseteq F$, for which $G_{F'/k'}$ has cyclic Sylow $p$-subgroups. The rank of $\overline{X}(A/L)$ and
the behaviour of the natural restriction and corestriction maps on $T(A/L)$ as $L$ varies over intermediate fields of $F/K$ has therefore a strong influence on the
explicit structure of the $\ZZ_p[G]$-module $\overline{X}(A/F)$.
Nevertheless, it is unlikely that this data alone will suffice to determine the structure of $\overline{X}(A/F)$ as a $\ZZ_p[G]$-module unless $P$ is cyclic. For example,
if $G=P$ is non-cyclic of order $p^2$, then there exist (non-isomorphic) abstract $\ZZ_p[G]$-lattices $M_1$ and $M_2$ such that for all subgroups $J'$ and
$J$ with $J'\le J \le G$ both $\rk_{\ZZ_p}(M_1^J) = \rk_{\ZZ_p}(M_2^J)$ and the groups $\hat H^{-1}(J/J',M_1^{J'})$ and $\hat H^{-1}(J/J',M_2^{J'})$ vanish and yet only
$M_1$ is a trivial source $\ZZ_p[G]$-module. (We are grateful to Griff Elder \cite{GE} for pointing out to us the existence of such examples).
\end{remark}

\subsubsection{The proof of Corollary \ref{general yak}} The key point here is that the given structure of $G$ implies that for any character $\rho$ in $\Ir(G)$ there
exists a subgroup $G_\rho$ of $G$ which contains $P$ and a linear character $\rho'$ of $G_\rho$ such that
$\rho = \Ind_{G_\rho}^G(\rho')$ (for a proof of this fact see \cite[II-22, Exercice]{serre} and the argument of \cite[II-18]{serre}).

We set $k_\rho := F^{G_\rho}$, $F_\rho:= F^{\ker(\rho')}$ and $K_\rho := F^J$ where $J$ contains $\ker(\rho')$ and is such that $J/\ker(\rho')$ is the Sylow
$p$-subgroup of $G_\rho/\ker(\rho')$. We also set $G_\rho' := G_\rho/\ker(\rho')$ and $P_\rho' := J/\ker(\rho')$. Then $k_\rho \subseteq F^P =K$ and hence
also $K_\rho \subseteq K$ since the degree of $K_\rho/k_\rho$ is prime to $p$, the extension $F_\rho/k_\rho$ is cyclic and the multiplicity with which $\rho$ occurs in
the $\QQ_p^c[G]$-module $\QQ_p^c\otimes X(A/F)$ is equal to the multiplicity
$m_{\rho'}$ with which $\rho'$ occurs in the $\QQ_p^c[G'_\rho]$-module $\QQ_p^c\otimes X(A/{F_\rho})$.

In addition, the result of Lemma \ref{sha tate2}(i) combines with the given hypotheses on $A$ and $F$ to imply that the pair $A_{/k_\rho}$ and $F_\rho$ satisfy the
hypotheses~\ref{hyp_a}--\ref{hyp_e}. From the proof of Theorem~\ref{yakovlev} we therefore know that knowledge of the diagram~\eqref{yak diag} with $F/k$ replaced by
$F_\rho/k_\rho$
determines the isomorphism class of the $\mathbb{Z}_p[G'_\rho]$-module $\overline{X}(A/{F_\rho})$ up to decompositions of the form~\eqref{yakovlev decomp}.
Such decompositions, with $M = \overline{X}(A/{F_\rho})$ say, determine $m_{\rho'}$ up to the multiplicity $m'_{\rho'}$ with which $\rho'$ occurs in the
scalar extension $\QQ_p^c\otimes M_2$ of the trivial source
$\mathbb{Z}_p[G'_\rho]$-module $M_2$.
We also write $m'_{\rho''}$ for the multiplicity with which the character $\rho''\in \Ir(P'_\rho)$ induced by the restriction of $\rho'$ to $J$
occurs in the $\mathbb{Q}_p^c[P'_\rho]$-module $\QQ_p^c\otimes M_2$. Then,
since $M_2$ is also a trivial source $\mathbb{Z}_p[P'_\rho]$-module and $P'_\rho$ is a $p$-group, the explicit structure of such modules
(and in particular the fact that $\ZZ_p[P'_\rho/H]$ is indecomposable for any subgroup $H$ of $P'_\rho$)
implies that $m'_{\rho'} \le m'_{\rho''} \le m'_{\eins_{P'_\rho}}$. Since
\[\QQ_p\otimes M_2^{P'_\rho} \subseteq \QQ_p\otimes X(A/{F_\rho})^{P'_\rho} \cong \QQ_p\otimes X(A/{K_\rho}) \subseteq \QQ_p\otimes X(A/{K})\] (with the
isomorphism a consequence of ~\eqref{rationalisom}), one therefore has
\[ m'_{\rho'} \le m'_{\eins_{P'_\rho}} = \dim_{\mathbb{Q}_p}\bigl(\QQ_p\otimes M_2^{P'_\rho}\bigr) \le \dim_{\mathbb{Q}_p}\bigl(\QQ_p\otimes X(A/{K})\bigr) = \srk(A/{K}).\]
This proves the first assertion of Corollary \ref{general yak}.

We next note that if $\Sel_p(A/K)$ is finite, then $\srk(A/{K})=0$ and so the above argument shows that the multiplicity $m_\rho$ with which any $\rho$
in $\Ir(G)$ occurs in the $\QQ_p^c[G]$-module $\QQ_p^c\otimes X(A/F)$ is uniquely determined  by a
  suitable diagram of the form ~\eqref{yak diag}. This implies the second assertion in Corollary \ref{general yak} because the $\QQ_p[G]$-module structure of
  $\QQ_p\otimes X(A/F)$ is determined up to isomorphism by the multiplicities $m_\rho$ for all $\rho$ in $\Ir(G)$. This completes the proof of Corollary
  \ref{general yak}.

\subsection{The proof of Theorem \ref{upper bound}}\label{proof of upper bound}

For each $\rho$ in $\Ir(P)$ we set $F_\rho := F^{\ker(\rho)}$. Then for each subgroup $H$ with $\ker(\rho)< H \le P$ one has
$\pi^F_{F^H} = \pi^{F_\rho}_{F^H}\circ \pi^F_{F_\rho}$ and so the module $T^{F_\rho}(A/{F^H})$ is a quotient of $T^F(A/{F^H})$. Thus if $T^F(A/{F^H})$ vanishes
for all $H\leq P$ (as assumed in claim (iii)), then $T^{F_\rho}(A/{F^H})$ vanishes for any given $\rho$ and $H$ as above (as assumed in claim (ii)).

Since the $\QQ_p^c[P]$-module $\QQ_p^c\otimes X(A/F)$ is isomorphic to $\bigoplus_{\rho \in \Ir(P)}V_\rho^{m_\rho}$ the  assertion of claim (iii) is therefore an immediate consequence of claim (ii). Indeed, if one has $m_\rho \le \rho(1)\cdot\srk(A/K)$ for all $\rho \in \Ir(P)$ (as would follow from claim (ii) under the hypotheses of claim (iii)), then
\[ \srk(A/F) = \sum_{\rho \in \Ir(P)}m_\rho\cdot\rho(1) \le \sum_{\rho \in \Ir(P)}\rho(1)^2\cdot\srk(A/K) = |P|\cdot \srk(A/K),\]
as required.

It therefore suffices to prove claims (i) and (ii) of Theorem \ref{upper bound} and to do this we argue by induction on $|P|$, using the result of
Lemma \ref{sha tate2}(ii).

To prove claim (i) by induction on $n$, it suffices to show that
\begin{multline*}
 \srk(A/{F_{n-1}}) + (p-1)\dim_{\mathbb{F}_p}\bigl(T^{F_n}(A/{F_{n-1}})\bigr) \le \srk(A/{F_n})\\
 \le p\cdot\srk(A/{F_{n-1}}) + (p-1)\dim_{\mathbb{F}_p}\bigl(T^{F_n}(A/{F_{n-1}})\bigr).
\end{multline*}
To do this we set $H := G_{F_n/F_{n-1}}$ and note that (as per the discussion in \S\ref{n=1} below) any $\ZZ_p[H]$-lattice is isomorphic to a direct sum of the form $N:= \ZZ_p^a\oplus \ZZ_p[H]^b \oplus R^c$ for suitable non-negative integers $a$, $b$ and $c$ and with $R:= \ZZ_p[H]/\bigl(\sum_{h\in H}h\bigr)$. This implies that $\rk(N^H) = a + b$ and $\rk(N) = a + pb + (p-1)c$ so $\rk(N^H) + (p-1)c \le \rk(N)$, with equality if and only if $b=0$, and $\rk(N) \le p\cdot \rk(N^H) + (p-1)c$, with equality if and only if $a=0$.
In addition, an easy computation shows that $\hat H^{-1}(H,R)\cong \mathbb{F}_p$ and so the group $\hat H^{-1}(H,N) \cong \hat H^{-1}(H,R)^c$ is an $\mathbb{F}_p$-space of dimension $c$. Given these observations for $N = \overline{X}(A/F)$ and the result of Lemma~\ref{sha tate2}(ii), the above displayed inequalities follow by applying Proposition~\ref{sha tate} with $J = H$.

Turning to the proof of claim (ii) we write $\overline{\rho}$ for the character of $P/\ker(\rho)$ that inflates to give $\rho$. Then, since $m_{F^{\ker(\rho)},\overline{\rho}} = m_{F,\rho}$, Lemma~\ref{sha tate2} allows us to replace $F$ by $F^{\ker(\rho)}$ and thus assume that $T^{F}(A/L) =  0 $ for all $L$ with $K \subseteq L\subsetneq F$. This is what we do in the rest of this argument.

We next note that $P$ is monomial and hence that for each $\rho$ in $\Ir(P)$ there is a non-trivial subgroup $J$ of $P$, a cyclic quotient $Q= J/J'$ of $J$ and a character $\psi$ in $\Ir(Q)$ with $\rho = \Ind_J^P\circ \Inf_{Q}^J(\psi)$. We set $F' := F^{J'}$ and $K' := F^J$.
Then $m_{F,\rho} = m_{F',\psi}$ and Lemma~\ref{sha tate2}(ii) implies that $A_{/k}$ and $F'$ satisfy hypotheses \ref{hyp_a}--\ref{hyp_e} with respect to $K'$. In
addition, for each $L$ with $K'\subseteq L\subsetneq F'$ the module $T^{F'}(A/{K'})$ is a quotient of $T^F(A/{K'})$ and so vanishes under our present hypotheses.
Since $Q = G_{F'/K'}$ is cyclic we may therefore apply Theorem~\ref{yakovlev} (ii) to $F'/K'$ in order to deduce that $\overline{X}(A/{F'})$ is a trivial source $\ZZ_p[Q]$-module.
From the explicit structure of a trivial source $\ZZ_p[Q]$-module it is then clear that $m_{F,\rho} = m_{F',\psi} \le m_{F',\eins_{Q}} = \srk(A/{K'})$. Since $\rho(1)$ is equal to $[P:J] = [K':K]$ it thus suffices to prove that $\srk(A/{K'}) \le [K':K]\cdot\srk(A/K)$.

To prove this we use the fact that, as $P$ is a $p$-group, there exists a finite chain of subgroups
\begin{equation}\label{chain}
 J = J_0 \trianglelefteq J_1 \trianglelefteq \cdots \trianglelefteq J_n = P
\end{equation}
in which $|J_{i+1}/J_i| < |P|$ for all $i$ with $0\le i < n$. For each such $i$ we set $F^i := F^{J_i}$. Then our hypotheses combine with Lemma~\ref{sha tate2}(ii) to imply that, for each $i$, $A_{/F^{i+1}}$ and $F^{i}$ satisfy the hypotheses of Theorem~\ref{upper bound}(iii) and hence, by induction, that $\srk(A/{F^{i}}) \le |J_{i+1}/J_{i}|\cdot \srk(A/{F^{i+1}}) = [F^{i}:F^{i+1}]\cdot \srk(A/{F^{i+1}}).$ It follows that
\begin{multline*} \srk(A/{K'}) = \srk(A/{F^{0}}) \le \biggl(\prod_{i=0}^{n-1}|J_{i+1}/J_{i}|\biggr) \cdot\srk(A/{F^{n}})\\ = [J_n:J_0]\cdot \srk(A/K) = [K':K]\cdot\srk(A/K)\end{multline*}
as required. This completes the proof of Theorem~\ref{upper bound}.

\subsection{The proof of Theorem \ref{noether}}\label{proof of noether}
 From (the proof of) Theorem~\ref{upper bound}(iii) we know that the conditions in Theorem~\ref{noether}(ii) and (iii) are equivalent. Next we note that if
 $\overline{X}(A/F)$ is a projective $\ZZ_p[G]$-module, then it is a free $\ZZ_p[P]$-module (by Swan's Theorem~\cite[(32.1)]{curtisr}), and hence also a free
 $\ZZ_p[H]$-module for each $H \le P$.
It follows that $\QQ_p\otimes X(A/F)$ is a free $\QQ_p[P]$-module and hence that both $\srk(A/F) = |H|\cdot \srk(A/{F^H})$ for each $H \le P$ and
$m_\rho = \rho(1)\cdot \srk(A/K)$ for each $\rho \in \Ir(P)$. In addition, the projectivity of $\overline{X}(A/F)$ implies that it is also cohomologically
trivial, and so Proposition~\ref{sha tate} implies that $T^{F}(A/{F^H}) =  0 $ for all $H \le P$.
 At this stage we have verified the implications~ (i) $\Longrightarrow$ (ii) $\Longleftrightarrow$ (iii) $\Longrightarrow$ (iv) $\Longrightarrow$ (v) in
 Theorem~\ref{noether} and so it suffices to prove (v) $\Longrightarrow$ (i).

To prove this implication we note that the $\ZZ_p[G]$-module $\overline{X}(A/F)$ is both finitely generated and $\ZZ_p$-free and hence
that~\cite[Chapter VI, (8.7), (8.8) and (8.10)]{brown} combine to imply that it is a projective $\ZZ_p[G]$-module if it is a cohomologically trivial $\ZZ_p[P]$-module.
We next claim that it is a cohomologically trivial $\ZZ_p[P]$-module if it is a free $\ZZ_p[C]$-module for each $C \le P$ with $|C| = p$.
The point here is that for any non-trivial subgroup $J$ of $P$ there exists a subgroup $C$ of $P$ which is normal in $J$ and has order $p$ and hence also a
Hochschild-Serre spectral sequence in Tate cohomology
$\hat H^{a}\bigl(J/C,\hat H^b(C,\overline{X}(A/F))\bigr) \Longrightarrow \hat H^{a+b}\bigl(J,\overline{X}(A/F)\bigr)$: thus if $\overline{X}(A/F)$ is
a free $\ZZ_p[C]$-module, then $\hat H^{m}\bigl(J,\overline{X}(A/F)\bigr)=  0 $ for all integers $m$, as required.
It therefore suffices to prove that the hypotheses in claim (v) imply that $\overline{X}(A/F)$ is a free $\ZZ_p[C]$-module for each subgroup $C$ of $P$ of order
$p$. Now, under the given conditions, the hypotheses of Theorem~\ref{yakovlev} (ii) are satisfied by the data $A_{/F^C}, F/F^C$ and so we know that
$\overline{X}(A/F)$ is a trivial source $\ZZ_p[C]$-module, and hence of the form $\ZZ_p[C]^m\oplus \ZZ_p^n$ for suitable non-negative integers $m$ and $n$.
It follows that $\srk(A/F) = p\,m+n$ and, using ~\eqref{rationalisom}, that $\srk(A/{F^C}) = m + n$ and so the equality $\srk(A/F) = p\cdot \srk(A/{F^C})$ in claim (v)
implies $n = 0$, and hence that $\overline{X}(A/F) = \ZZ_p[C]^m$ is a free $\ZZ_p[C]$-module, as required.

This completes the proof of Theorem~\ref{noether}.

\subsection{The proof of Corollary \ref{noether cor}}\label{proof of noether cor}
We first note that, under the assumed validity of hypotheses~\ref{hyp_a}--\ref{hyp_e}, Proposition~\ref{sha tate} shows that for every subgroup $J$ of $P$
the natural homomorphism $X(A/F)_J \to X(A/{F^J})$ is bijective.

In particular, if we now assume that $\overline{X}(A/F)$ is a projective $\ZZ_p[G]$-module, then there is a direct sum decomposition of $\ZZ_p[G]$-modules
$X(A/F) \cong T(A/F) \oplus \overline{X}(A/F)$ and $\overline{X}(A/F)_J$ is $\ZZ_p$-free and so the above isomorphism induces an isomorphism
$T(A/{F})_J \cong T(A/{F^J})$. Now, under the assumed vanishing of $\sha(A/K)_p$ the module $T(A/K) \cong T(A/{F})_P$ vanishes and so Nakayama's Lemma implies
$T(A/{F})$, and hence also $T(A/{F^J})$ for each $J \le P$, vanishes. Taken in conjunction with the proof of Theorem~\ref{noether} given above, this verifies the
implications~ (iv) $\Longleftrightarrow$ (v) $\Longleftrightarrow$ (i) $\Longrightarrow$ (ii) $\Longleftrightarrow$ (iii) in Corollary~\ref{noether cor} and hence reduces
us to proving that~ (ii) $\Longrightarrow$ (v) under the stated hypotheses.

To prove this implication it suffices to prove that, under the hypotheses of Corollary~\ref{noether cor}, if $\QQ_p\otimes X(A/F)$ is a free $\QQ_p[P]$-module, then
$T(A/{F})$ vanishes and $X(A/F)$ is a cohomologically trivial $\ZZ_p[C]$-module for each subgroup $C$ of $P$ of order $p$.  To prove this we use induction on
$|P|$. We thus fix a subgroup $C$ of $P$ of order $p$ and then choose a chain of subgroups as in~\eqref{chain} but with $J$ replaced by $C$. Then Lemma~\ref{sha tate2}
shows that each set of data $A_{/F^{i+1}}$ and $F^i$ satisfies~\ref{hyp_a}--\ref{hyp_e} and hence, by induction, we can deduce that $T(A/{F^C}) =  0 $. By
Proposition~\ref{sha tate}, we have that $\hat H^{-1}\bigl(C,\overline{X}(A/F)\bigr) =  0 $.
Since $\QQ_p\otimes X(A/F)$ is by assumption a free $\QQ_p[C]$-module, a Herbrand quotient argument then implies that
$\hat H^{0}\bigl(C,\overline{X}(A/F)\bigr) =  0 $ and hence that $\overline{X}(A/F)$ is a cohomologically trivial $\ZZ_p[C]$-module (since, for example, the
Tate cohomology of $C$ is periodic of order two).
By~\cite[Chapter VI, (8.7)]{brown} we may now deduce that $\overline{X}(A/F)$ is a projective $\ZZ_p[C]$-module and then the same argument as used at the beginning
of this section shows that $T(A/{F^C}) =  0 $ implies $T(A/{F}) =  0 $. This completes the proof of Corollary~\ref{noether cor}.

\subsection{The proofs of Corollaries \ref{mr cor} and \ref{ftct cor}}\label{last}

To complete the proof of all of the results stated in~\S\ref{mwts} it now only remains to prove Corollaries \ref{mr cor} and \ref{ftct cor}.

\subsubsection{The proof of Corollary \ref{mr cor}} The given hypotheses imply that the main results of Mazur and Rubin in~\cite{mr2} are valid.
Their Theorem~B thus shows that the $\QQ_p[P]$-module $\QQ_p\otimes X(A/F)$ has a direct summand that is isomorphic to $\QQ_p [P]$. Given this fact, claim (ii) follows
immediately from the equivalence of claims (i) and (iii) in Corollary~\ref{noether cor}, whilst claim (i) is a straightforward consequence of
Theorem~\ref{yakovlev} (ii) and the fact that if $P$ is cyclic then any trivial source $\ZZ_p[P]$-module $M$ for which the associated $\QQ_p[P]$-module $\QQ_p\otimes M$ has a free
rank one direct summand must itself have a direct summand that is isomorphic to $\ZZ_p[P]$.

\subsubsection{The proof of Corollary \ref{ftct cor}}

The first thing to note is that, since the extension $\QQ(\zeta_{p^n})/\QQ(\zeta_p)$ is a $p$-extension, the given hypotheses imply (via Lemma \ref{sha tate2}) that the abelian variety
$A$ satisfies the hypotheses~\ref{hyp_a}--\ref{hyp_e} with $F = K_d^{n_d,n}, K = \QQ(\zeta_{p^n})$ and $k$ equal to the subextension of $\QQ(\zeta_{p^n})$ of degree $p^{n-1}$ over $\QQ$
(so that $G = G_{n}$ and $P = P_{n}$).

Given this, claim (i) follows directly from Theorem \ref{yakovlev} (ii). In order to prove claim (ii), we first note that any non-linear character $\rho$ of
$G$ is of the form $\Ind^G_P(\psi)$ for a linear character $\psi$ of $P$. From the explicit structure of a trivial source module over $\ZZ_p[P]$ it is then clear that
$m_\rho=m_\psi\leq m_{\eins_{P}} = \srk\bigl(A/{\QQ(\zeta_{p^n})}\bigr)$. It is also clear that, if $\rho'=\Ind^G_P(\psi')$ for a linear character
$\psi'$ of $P$, then $m_\psi\leq m_{\psi'}$ if and only if $\ker(\psi)\subseteq \ker(\psi')$.

Similarly, by using the fact that $\overline{X}(A/{K_d^{n_d,n}})$ is a trivial source module in the setting of claim (iii), one finds that
$m_\rho\leq \rho(1)\,m_\phi$ for any $\rho$ in $\Ir(G)$ whose restriction to $H$ has $\phi$ occurring with non-zero multiplicity. In this inequality $m_\phi$ denotes
the multiplicity with which $\phi$ occurs in the module $\QQ^c_p\otimes X\bigl(A/{\QQ(\zeta_{p^n})}\bigr)=\QQ^c_p\otimes X(A/{K_d^{n_d,n}})^P$. By arguing as in
the proof of Theorem \ref{upper bound}(iii) it then follows that
$\rk_{\ZZ_p}\bigl(e_\phi \cdot X(A/{K_d^{n_d,n}})\bigr) \le p^{n_d}\cdot \rk_{\ZZ_p}\bigl(e_\phi \cdot X(A/{\QQ(\zeta_{p^n})})\bigr)$, as required to complete the proof
of claim (iii).

\section{Examples}\label{examples} To end the article we discuss several concrete examples in which the classification results of Heller and Reiner in \cite{hr} can be combined with the results of \S\ref{ago} to make the result of Theorem \ref{yakovlev} much more explicit.

To do this we assume to be given a (possibly finite) pro-cyclic pro-$p$ extension of number fields $F/k$ and an abelian variety $A$ over $k$ which satisfies the hypotheses~\ref{hyp_a}--\ref{hyp_e} with respect to the field $K =k$. We set $G:= G_{F/k}$ and for each non-negative integer $i$ we write $F_i$ for the unique field with $k \subseteq F_i \subseteq F$ and $[F_i:k]$ equal to $p^i$ unless $p^i > |G|$ in which case we set $F_i = F$. For each such $i$ we then set $r_i := \srk(A/{F_{i}})$ and define an integer $t_i$ by the equality $p^{t_i} :=  |T^F\!(A/{F_{i}})|$. We also write $\iota(G)$ for the number of isomorphism classes of indecomposable $\ZZ_p[G]$-lattices.

For each non-negative integer $i$ we write $R_i$ for the quotient $\ZZ_p[\ZZ/(p^i)]/(T)$ where $T$ denotes the sum in $\ZZ_p[\ZZ/(p^i)]$ of all elements of $(p^{i-1})/(p^i)\subset \ZZ/(p^i)$. For any integer $i$ such that $G$ has a quotient $G_i$ of order $p^i$ we regard $R_i$ as a $\ZZ_p[G]$-module by means of the homomorphism $\ZZ_p[G]\to \ZZ_p[G_i]\cong \ZZ_p[\ZZ/(p^i)] \to R_i$ where the isomorphism is induced by choosing a generator of $G_i$ (the precise choice of which will not matter in what follows) and the unlabelled arrows are the natural projection maps.

\subsection{}\label{n=1}$|G|=p$. In this case $\iota(G) = 3$ (by \cite[Theorem 2.6]{hr}) with representative modules $\ZZ_p, \ZZ_p[G]$ and $R_1$. In addition, the groups $\hat H^{-1}(G,\ZZ_p)$ and $\hat H^{-1}(G,\ZZ_p[G])$ vanish and $\hat H^{-1}(G,R_1)$ has order $p$. In particular, if one sets $\delta:= (r_0-r_1)/(p-1)$, then a comparison of ranks and of Tate cohomology shows (via Proposition \ref{sha tate}) that $t_1\le \delta$ and that there is an isomorphism of $\ZZ_p[G]$-modules
\[ \overline{X}(A/F)\cong \ZZ_p^{r_1+t_1-\delta}\oplus  \ZZ_p[G]^{\delta-t_1} \oplus R_1^{t_1}.\]

 \subsection{}$|G| = p^2$. In this case $\iota(G) = 4p+1$ (by \cite[\S 4]{hr}) but one finds that the group $\hat H^{-1}(G,M)$ only vanishes for modules $M$ in an explicit subset $\Upsilon$ of $p+2$ of these isomorphism classes (see, for example, Table 2 in \cite{RVM}). Hence, if $\sha(A/k)_p$ vanishes, then $t_2=0$ and Proposition \ref{sha tate} implies $\overline{X}(A/F)$ is a direct sum of modules from $\Upsilon$. If one further assumes for example that $r_1=r_2$, then a comparison of ranks and Tate cohomology groups (over the subgroup of $G$ of order $p$) of the modules in $\Upsilon$ shows the existence of an integer $s_1$ with both $(p-1)s_1 = t_1$ and $s_1\le r_2$ and such that there is an isomorphism of $\ZZ_p[G]$-modules
\[ \overline{X}(A/F)\cong \ZZ_p^{r_2-s_1} \oplus (R_2,\ZZ_p,1)^{s_1}\]
where the indecomposable module $(R_2,\ZZ_p,1)$ is  an extension of  $R_2$ by $\ZZ_p$ which corresponds to the image of $1$ in $\Ext^1_{\ZZ_p[G]}(R_2,\ZZ_p)\cong \ZZ/p$.

\subsection{}$|G| > p^2$. In this case $\iota(G)$ is infinite (as proved in \cite{hr2}) but Proposition \ref{sha tate} still imposes very strong restrictions on those indecomposable modules that can occur in the decomposition of the modules $\overline{X}(A/F_i)$. For example, if $\Sel_p(A/k)$ vanishes, then the isomorphism $X(A/F_i)_{G_{i}} \cong X(A/k)$ coming from Proposition \ref{sha tate} combines with Nakayama's Lemma to imply that $\Sel_p(A/F_i)$ vanishes for all $i\ge 0$.

In the following result we describe the next simplest case.

\begin{proposition}\label{procyclic} Assume that the abelian variety $A$ and the pro-cyclic extension $F/k$ satisfy all of the above hypotheses. Assume in addition that $\Sel_p(A/k)$ is both finite and non-trivial. Then either

\begin{itemize}
\item[(i)] $\sha(A/F_i)_p$ is non-trivial for all $i \ge 0$, or
\item[(ii)] there exists a natural number $n$, a strictly positive integer $m_n$ and for each integer $i$ with $1\le i < n$ a non-negative integer $m_i$ such that the following conditions are satisfied.
\begin{itemize}
\item[(a)] $\sha(A/F_n)_p$ is trivial;
\item[(b)] There is a decreasing filtration of $X(A/F)$ with associated graded object $\bigoplus_{i=0}^{n-1}R_{n-i}^{m_{n-i}}$. In particular, for each integer $i$ with $1\le i \le n$ one has $\srk(A/F_i) = \sum_{a=1}^{i}m_a(p^a-p^{a-1})$;
\item[(c)] There is a decreasing filtration of $X(A/k)$ with associated graded object $\bigoplus_{i=0}^{n-1}(\ZZ/p)^{m_{n-i}}$. In particular, the order of $\Sel_p(A/k)$ is $\prod_{a=1}^{n}p^{m_a}$ and its exponent divides $p^e$ with $e$ equal to the number of integers $m_a$ that are strictly positive.
\end{itemize}
\end{itemize}
\end{proposition}

 \begin{proof} It is enough to assume the existence of a natural number $m$ for which $\sha(A/F_m)_p$, and hence also $T(A/F_m)$, is trivial and then use this to prove all assertions in claim (ii). We write $n$ for the least such $m$ and note that, since $\sha(A/F_n)_p$ is trivial, for each integer $i$ with $0\le i < n$ the group $\sha(A/F_i)_p$ is finite and hence equal to $T(A/F_i)$.

 We may also assume henceforth that $F = F_n$ so that $G=G_n$. We then set $\overline{X} := \overline{X}(A/F) = X(A/F)$ and for each integer $i$ with $0\le i \le n$ we write $G^i$ for the subgroup of $G$ of order $p^{i}$ and $G_i$ for the quotient of $G$ of order $p^i$ (so that $G_i \cong G/G^{n-i}$).

 Then for each integer $i$ with $0 \le i < n$ there exists a non-negative integer $m_{n-i}$ and a short exact sequence of $\ZZ_p[G]$-modules of the form
\begin{equation}\label{key ses}0 \to \overline{X}^{G^{i+1}} \to \overline{X}^{G^i} \to R_{n-i}^{m_{n-i}}\to 0.\end{equation}
In particular, since the finiteness of $\Sel_p(A/k)$ implies that $\overline{X}^{G^n}$ vanishes, the $\ZZ_p[G]$-module $\overline{X}$ has a decreasing filtration
\[ \{0\} = \overline{X}_n \subseteq \overline{X}_{n-1} \subseteq \dots \subseteq \overline{X}_0 = \overline{X}\]
with each $\overline{X}_i := \overline{X}^{G^i}$ a $\ZZ_p[G_{n-i}]$-module and each quotient $\overline{X}_i/\overline{X}_{i+1}$ isomorphic to $R_{n-i}^{m_{n-i}}$. This gives the filtration on $\overline{X}=X(A/F)$ described in claim~(ii)(b) and also implies that the  $\QQ_p[G]$-module $\QQ_p\otimes \overline{X}$ is isomorphic to $\bigoplus_{i=1}^{i=n}(\QQ_p\otimes R_i)^{m_i}$.

Now for each integer $a$ with  $1\le a \le n$ and each subgroup $H$ of $G$ the module of invariants $R_a^H$ is equal to $R_a$ if $H\subseteq G^{n-a}$ and otherwise vanishes. For each integer $i$ with  $1\le i \le n$, by using ~\eqref{rationalisom}, one therefore has
\[ \srk(A/F_i) = \dim_{\QQ_p}\Bigl((\QQ_p\otimes \overline{X})^{G^{n-i}}\Bigr) = \sum_{a=1}^{a=i}m_a\cdot \dim_{\QQ_p}(\QQ_p\otimes R_a) = \sum_{a=1}^{a=i}m_a\cdot (p^a-p^{a-1})\]
where the last equality follows from the fact that for each $a$ there is a short exact sequence of $\ZZ_p[G]$-modules
\begin{equation}\label{key ses2} 0 \to \ZZ_p[G_{a-1}] \to \ZZ_p[G_a] \to R_a\to 0\end{equation}
and hence that $\dim_{\QQ_p}(\QQ_p\otimes R_a) = |G_a| - |G_{a-1}| = p^a-p^{a-1}$. This proves claim~(ii)(b).

To prove claim (ii)(c) the key point is that Proposition \ref{sha tate} combines with the vanishing of $T(A/F)$ and the finiteness of $\Sel_p(A/k)$ to give an isomorphism of finite abelian groups
\begin{equation}\label{key iso} \hat H^{-1}(G,\overline{X})\cong T^{F}(A/k) = T(A/k) = X(A/k) \end{equation}
whilst the exact sequences~\eqref{key ses} and~\eqref{key ses2} provide a means of explicitly computing $\hat H^{-1}(G,\overline{X})$. More precisely, since the module of invariants $\overline{X}^G$ vanishes, as does the module $R^G_{n-a}$ for each integer $a$ with $0 \le a < n$, and since Tate cohomology with respect to $G$ is periodic of order two, the exact sequences~\eqref{key ses} give rise to associated short exact sequences
\begin{equation*}\label{key filt} 0 \to \hat H^{-1}(G,\overline{X}^{G^{a+1}}) \to \hat H^{-1}(G,\overline{X}^{G^a}) \to \hat H^{-1}(G,R_{n-a})^{m_{n-a}}\to 0\end{equation*}
whilst the exact sequence~\eqref{key ses2} gives rise to an isomorphism
\begin{multline*} \hat H^{-1}(G,R_a) = \hat H^{-1}(G_a,R_a)\cong \hat H^0\bigl(G_a,\ZZ_p[G_{a-1}]\bigr)\\ \cong \hat H^0\bigl(G^{n-a+1}/G^{n-a},\ZZ_p\bigr) \cong \ZZ/p. \end{multline*}
By combining the above exact sequences with the isomorphism~\eqref{key iso}, we thereby obtain a decreasing filtration of $X(A/k)$ of the sort described in claim~(ii)(c).

Given this filtration, it is then easy to deduce the assertions in claim~(ii)(c) concerning the order and exponent of $\Sel_p(A/k)$.

At this stage it only remains to prove that the integer $m_n$ is strictly positive. However, if this is not the case, then the sequence~\eqref{key ses} with $i = 0$ implies that $G^{1}$ acts trivially on the module $\overline{X}$ and hence that the group $\hat H^{-1}(G^{1},\overline{X})$ vanishes. Moreover, Lemma \ref{sha tate2} implies that the data $A$ and $F/F_{n-1}$ satisfy the hypotheses~\ref{hyp_a}--\ref{hyp_e} and hence, by applying Proposition \ref{sha tate} to this data, we deduce that the group $T^{F}(A/F_{n-1})$ also vanishes. In particular, since $T(A/F)$ vanishes the group $T(A/F_{n-1}) = \sha(A/F_{n-1})_p$ must also vanish and this contradicts the minimality of $n$. \end{proof}

\begin{example}
 We give some simple examples showing that in certain situations the result of Proposition \ref{procyclic} becomes very explicit. To do this we assume
the notation and hypotheses of Proposition \ref{procyclic}. We also assume that $\Sel_p(A/k)$ has order $p^2$ and that $T(A/F_m)$ vanishes for some integer $m$. In
this case Proposition \ref{procyclic}(ii)(c) implies that the least such $m$ is greater than or equal to two and we assume, for simplicity, that this integer is equal to
two. We then also assume that $F=F_2$ and hence that $G = G_2$.

\noindent{}(i) If $\Sel_p(A/k)$ is cyclic of order $p^2$, then Proposition \ref{procyclic}(ii)(c) implies that $m_1 = m_2 = 1$ and hence that
$X(A/F_2)$ is an extension of $R_2$ by $R_1$. Moreover, the Heller-Reiner classification of $\ZZ_p[G]$-lattices shows (via Table 2 in \cite{RVM}) that, up to
isomorphism, there is a unique such extension $M$ for which $\hat H^{-1}(G,M)$ is isomorphic to $\ZZ/p^2$ (rather than having exponent dividing $p$). Putting everything
together one finds that in this case $X(A/F_2)$ is isomorphic as a $\ZZ_p[G]$-module to the indecomposable module $(R_2,R_1,1)$ which corresponds to the image of
$1$ in $\Ext^1_{\ZZ_p[G]}(R_2,R_1)\cong \ZZ/p$. Using the isomorphism $X(A/F_1) \cong X(A/F_2)_{G^1} \cong (R_2,R_1,1)_{G^1}$ one then computes that the
$\ZZ_p[G]$-modules $T(A/F_1)$ and $\overline{X}(A/F_1)$ are isomorphic to $\ZZ/p$ and $R_1$ respectively.

\noindent{}(ii) If $\Sel_p(A/k)$ has exponent $p$ and order $p^2$, then Proposition \ref{procyclic}(ii)(c) implies that either $m_2 = 2$ and $m_1 = 0$ or that
$m_2=m_1 = 1$. In the former case $X(A/F_2)$ is isomorphic to $R_2^2$ and hence $X(A/F_1)\cong X(A/F_2)_{G^1}$ is isomorphic to
$(R_2^2)_{G^1} \cong (R_{2,G^1})^2 \cong (\ZZ/p)^{2p}$. In the latter case the Heller-Reiner classification implies that $X(A/F_2)$ is isomorphic to either
$R_1\oplus R_2$ or, for some integer $i$ with $0 < i< p$, to the indecomposable $\ZZ_p[G]$-module $(R_2,R_1,i)$ which corresponds to the image of the polynomial
$(1-x)^i$ in the group $\Ext_{\ZZ_p[G]}^1(R_2,R_1) \cong R_1/p \cong (\ZZ/p)[x]/(1-x)^{p-1}$.

If $X(A/F_2)\cong R_1\oplus R_2$, then one computes that $X(A/F_1)\cong X(A/F_2)_{G^1}$ is isomorphic as a $\ZZ_p[G]$-module to $(\ZZ/p)^p\oplus R_1$.
In a similar way, if $X(A/F_2)\cong (R_2,R_1,i)$ with $0< i< p$, then one computes that the abelian group
$T(A/F_1) = T^F(A/F_1)\cong \hat H^{-1}\bigl(G^1,X(A/F_2)\bigr)$ is isomorphic to $(\ZZ/p)^{i+1}$ and that the $\ZZ_p[G]$-module $\overline{X}(A/F_1)$ is
isomorphic to  $R_1$.
\end{example}

\subsection*{Acknowledgements} It is a pleasure for us to thank Werner Bley and Griff Elder for helpful discussions and correspondence concerning this work,
and the anonymous referees for making several useful suggestions and for correcting an error in one of our arguments.


\bibliographystyle{amsplain}
\bibliography{sel}

\end{document}